\documentclass[12pt,a4paper,leqno]{article}

\usepackage[nottoc]{tocbibind}
\usepackage{latexsym}
\usepackage{amssymb,exscale}
\usepackage[centertags]{amsmath}
\usepackage{amsthm}
\usepackage{footnpag}

\usepackage[linktocpage]{hyperref}

\numberwithin{equation}{subsection}

\swapnumbers
\theoremstyle{definition}
\newtheorem{theorem}[equation]{Theorem}
\newtheorem{lemma}[equation]{Lemma}

\newtheorem{proposition}[equation]{Proposition}

\newtheorem*{theorem1}{Theorem 1}

\renewcommand{\phi}{\varphi}

\newcommand{\I}{{\rm i}}

\newcommand{\E}{\mathrm{e}}

\newcommand{\ti}{\tilde}

\renewcommand{\(}{\bigl(}
\renewcommand{\)}{\bigr)\vphantom{)}}

\renewcommand{\equiv}{\;\>\Longleftrightarrow\;\>}
\newcommand{\equ}{$\;\>\Longleftrightarrow\;\>$}
\newcommand{\imply}{\;\;\;\Longrightarrow\;\;\;}
\newcommand{\imp}{$ \Longrightarrow $ }

\newcommand{\Atoms}{\operatorname{Atoms}}
\newcommand{\Down}{\operatorname{Down}}
\newcommand{\Up}{\operatorname{Up}}
\newcommand{\Cl}{\operatorname{Cl}}

\newcommand{\One}{{1\hskip-2.5pt{\rm l}}}

\newcommand{\Si}{\Sigma}
\newcommand{\si}{\sigma}

\newcommand{\Om}{\Omega}

\newcommand{\al}{\alpha}

\newcommand{\M}{\mathcal M}

\newcommand{\F}{\mathcal F}
\newcommand{\A}{\mathcal A}
\newcommand{\B}{\mathcal B}

\newcommand{\La}{\Lambda}

\newcommand{\Q}{\mathbf{Q}}

\newcommand{\Ex}{\mathbb E\,}
\newcommand{\R}{\mathbb R}

\renewcommand{\Pr}[1]{\mathbb{P}\mskip1.5mu\(\mskip1.5mu#1\mskip1.5mu\)}

\newcommand{\sif}{$\sigma$\nobreakdash-field}

\newcommand{\valued}[1]{$#1$\nobreakdash-\hspace{0pt}valued}
\newcommand{\almost}[1]{$#1$\nobreakdash-\hspace{0pt}almost}

\newcommand{\measurable}[1]{$#1$\nobreakdash-\hspace{0pt}measurable}

\begin{document}

\title{Noise as a Boolean algebra of \sif s. III.\\
 An old question of Jacob Feldman}

\author{Boris Tsirelson}

\date{}
\maketitle

\begin{abstract}
The noise-type completion $ C $ of a noise-type Boolean algebra $ B $
is generally not the same as the closure $ \ti B $ of $ B $. As shown
in Part I (Introduction, Theorem 2), $ C $ consists of all
\emph{complemented} elements of $ \ti B $. It appears that $ C = \ti B
$ if and only if $ B $ is classical (as defined in Part II, Sect.~1a),
which is the main result of this Part III.
\end{abstract}

\setcounter{tocdepth}{2}
\tableofcontents

\section*{Introduction}
\addcontentsline{toc}{section}{Introduction}
According to \cite{Ts1}, the noise-type completion $ C $ of a
noise-type Boolean algebra $ B $ consists of all complemented elements
of the closure of $ B $. (The monotonic closure $ \ti B $ and the
topological closure $ \Cl(B) $ are the same, see \cite[Sect.~1d,
especially 1d6 and 1d9]{Ts1}.) Are there any non-complemented elements
in $ \Cl(B) $? That is, whether $ C = \Cl(B) $, or not? This question
was not addressed in \cite{Ts1}, \cite{Ts2}.

The definition of a factored probability space, given by J.~Feldman
\cite{Fe} in 1971, is equivalent to our ``noise-type Boolean algebra''
with an additional requirement: $ B $ must be monotonically
closed. This is, $ B = C = \Cl(B) $. In that framework, Feldman asked
a question \cite[Problem 1.9]{Fe}: is every factored probability space
linearizable? His linearizability is equivalent to our classicality
\cite[Def.~1a2]{Ts2}. In our terms, Feldman gives a detailed
description of classical noise-type Boolean algebras, and leaves open
the question, whether the equality $ C = \Cl(B) $ implies
classicality, or not.

\begin{theorem1}\label{th1}
Let $ B $ be a noise-type Boolean algebra. Denote by $ C $ its
noise-type completion, and by $ \Cl(B) $ its closure. The following
two conditions are equivalent:

(a) $ C = \Cl(B) $;

(b) $ B $ is classical;

(c) $ \sup_n x_n \in C $ for all $ x_1,x_2,\dots \in B $.
\end{theorem1}

This is basically a remake of \cite[Sect.~6c/6.3]{Ts03}. The
implication (a) \imp (c) is trivial: $ \sup_n x_n = \lim_n
(x_1\vee\dots\vee x_n) \in \Cl(B) = C $ (but the converse implication
is nontrivial, since in general an element of $ \Cl(B) $ is not a
supremum of elements of $ B $). The implication (b) \imp (a), proved
in Section \ref{sec:1}, is relatively easy. The implication (c) \imp
(b), proved in Sections \ref{sec:2} and \ref{sec:3}, is more
difficult.

\section[The easy direction]{\raggedright The easy direction}
\label{sec:1}
\subsection{Preliminaries}
\label{1a}

Type $ L_2 $ subspaces of $ H = L_2(\Om,\F,P) $ are defined in
\cite[Sect.~1a]{Ts1}. For every (linear, closed) subspace $ h \subset
H $, the subspace $ \ti h = L_2 \( \si(h) \) $ is the type $ L_2 $
subspace generated by $ h $. Here $ \si(h) $ denotes the
\sif\ generated by $ h $ (that is, by all elements of $ h $).

Let $ h, h_1, h_2, \dots $ be subspaces of $ H $, and $ \ti h, \ti
h_1, \ti h_2, \dots $ the corresponding type $ L_2 $ subspaces; claim:
\begin{equation}\label{1a1}
\text{if } h_n \uparrow h \text{ then } \ti h_n \uparrow \ti h \, .
\end{equation}
Here $ h_n \uparrow h $ means that $ h_1 \subset h_2 \subset \dots $
and $ h $ is the closure of $ \cup_n h_n $; similarly, $ \ti h_n
\uparrow \ti h $ means that $ \ti h_1 \subset \ti h_2 \subset \dots $
and $ \ti h $ is the closure of $ \cup_n \ti h_n $. The claim holds
since $ \si(h_n) \uparrow \si(h) $ implies $ L_2 \( \si(h_n) \)
\uparrow L_2 \( \si(h) \) $ (as noted in \cite[proof of 1d2]{Ts1}).

Likewise, $ \si(h_n) \downarrow \si(h) $ implies $ L_2 \( \si(h_n) \)
\downarrow L_2 \( \si(h) \) $, but nevertheless, the relation $ h_n
\downarrow h $ (that is, $ h_1 \supset h_2 \supset \dots $ and $ h =
\cap_n h_n $) does not imply $ \ti h_n \downarrow \ti h $. A
counterexample: $ H = L_2(0,2\pi) $, and $ h_n $ is spanned by the
functions $ t \mapsto \E^{\I kt} $ for $ k = n,n+1,\dots $; then $ h_n
\downarrow \{0\} $, however, $ \ti h_n = H $ for all $ n $, since $
\ti h_n $ contains the function $ t \mapsto \overline{ \E^{\I nt} }
\cdot \E^{\I(n+k)t} = \E^{\I kt} $.

Another remark. Let $ X,X_1,X_2,\dots $ and $ Y,Y_1,Y_2,\dots $ are
random variables on a given probability space, and for every $ n $ the
two random variables $ X_n, Y_n $ are independent; claim:
\begin{equation}\label{1a2}
\text{if } X_n \to X, \> Y_n \to Y \text{ in probability then } X,Y
\text{ are independent.}
\end{equation}
Proof: if $ f,g : \R \to \R $ are bounded continuous functions then $
\Ex(f(X_n)) \to \Ex(f(X)) $, $ \Ex(g(Y_n)) \to \Ex(g(Y)) $,  $
\Ex(f(X_n)) \Ex(g(Y_n)) = \Ex(f(X_n)g(Y_n)) \to \Ex(f(X)g(Y)) $, thus,
$ \Ex(f(X)g(Y)) = \Ex(f(X)) \Ex(g(Y)) $. The same holds for
vector-valued random variables.

\subsection{Reminder}
\label{1b}

Let $ B $ be a noise-type Boolean algebra \cite[Def.~2a1]{Ts1} of \sif
s on a probability space $ (\Om,\F,P) $. Recall that
\begin{gather}
B \text{ is a sublattice of a complete lattice } \La \, ; \\
\La \text{ is endowed with a topology} \, ; \\
x_n \downarrow x \text{ implies } x_n \to x \, ; \; \text{ also }
 x_n \uparrow x \text{ implies } x_n \to x \, ; \quad
 \text{(see \cite[1d2]{Ts1})} \label{1b3}
\end{gather}
here $ x, x_1, x_2, \dots \in \La $, and
\begin{gather*}
x_n \downarrow x \quad \text{means} \quad x_1 \ge x_2 \ge \dots \text{
  and } \inf_n x_n = x \, , \\
x_n \uparrow x \quad \text{means} \quad x_1 \le x_2 \le \dots \text{
  and } \sup_n x_n = x \, .
\end{gather*}
Thus,
\begin{gather}
\text{every monotone sequence of } x_n \in \La \text{ converges in }
 \La \, ; \label{1b4} \\
\text{every monotone sequence of } x_n \in \Cl(B) \text{ converges in
} \Cl(B) \, ; \label{1b5}
\end{gather}
here $ \Cl(B) $ is the closure of $ B $ in $ \La $.
% In this text (Part III) we seldom use elements of $ \La \setminus \Cl(B) $.

Recall the noise-type completion $ C $ of $ B $;
\begin{equation}\label{1b6}
B \subset C \subset \Cl(B) \, , 
\end{equation}
and $ C $ is the greatest noise-type Boolean algebra satisfying
\eqref{1b6}. Also,
\begin{equation}\label{1b7}
C \text{ is the set of all complemented elements of } \Cl(B) \, ;
\end{equation}
here $ x \in \Cl(B) $ is called complemented, if there exists
(necessarily unique) $ x' \in \Cl(B) $ (called the complement of $ x
$) satisfying $ x \wedge x' = 0 $ and $ x \vee x' = 1 $.

Let $ x,y \in \La $ be independent and $ u \le x $, $ v \le y $;
claim:
\begin{equation}\label{1b8}
\text{if } u \vee v = 1 \quad \text{then } u=x, \, v=y \, .
\end{equation}
This is a simple special case of \cite[1d13]{Ts1}.

Recall also \sif s $ \F_x $ \cite[1a]{Ts1}, projections $ Q_x : H \to
H $ \cite[1d]{Ts1}, and the first chaos space $ H^{(1)} \subset H $
\cite[1a2]{Ts2}. The space $ H^{(1)} $ is invariant under projections
$ Q_x $ for $ x \in \Cl(B) $ (since $ Q_{x\vee y} Q_z \psi = Q_z
Q_{x\vee y} \psi = Q_z ( Q_x \psi + Q_y \psi ) = ( Q_x + Q_y ) Q_z
\psi $).

\subsection{Proof that (b) implies (a)}
\label{1c}

We assume that $ B $ is classical and prove that $ C = \Cl(B) $, thus
proving the implication (b) \imp (a) of Theorem 1.

We denote by $ \Down(x) $, for $ x \in \Cl(B) $, the restriction of $
Q_x $ to $ H^{(1)} $ (treated as an operator $ H^{(1)} \to H^{(1)}
$). Clearly, $ x \le y $ implies $ \Down(x) \le \Down(y) $. By
\cite[(1a3)]{Ts2},
\begin{equation}\label{1c1}
\Down(x) + \Down(x') = \One \quad \text{for } x \in B \, .
\end{equation}

We denote by $ \Q $ the closure of $ \{ \Down(x) : x \in B \} $ in the
strong operator topology; $ Q $ is a closed set of commuting
projections on $ H^{(1)} $; we have $ \Down(x) \in \Q $ for $ x \in B
$, and by continuity for $ x \in \Cl(B) $ as well.

Note that $ q \in \Q $ implies $ \One-q \in \Q $ (since $ \Down(x_n)
\to q $ implies $ \Down(x'_n) \to \One-q $).

We define $ \Up(q) $, for $ q \in \Q $, as $ x \in \La $ such that $
\F_x = \si(qH^{(1)}) $ (the \sif\ generated by $ q\psi $ for all $
\psi \in H^{(1)} $). Clearly, $ q_1 \le q_2 $ implies $ \Up(q_1) \le
\Up(q_2) $. Also,
\begin{equation}\label{1c2}
\Up(q) \vee \Up(\One-q) = 1 \quad \text{for } q \in Q \, ,
\end{equation}
since $ qH^{(1)} + (\One-q)H^{(1)} = H^{(1)} $, and $ \si(H^{(1)}) = \F_1
$ by the classicality.

\begin{lemma}\label{1c3}
$ \Up(q) $ and $ \Up(\One-q) $ are independent (for each $ q \in \Q $).
\end{lemma}

\begin{proof}
We take $ x_n \in B $ such that $ \Down(x_n) \to q $, then $
\Down(x'_n) \to \One-q $. We have to prove that $ \si(qH^{(1)}) $ and $
\si((\One-q)H^{(1)}) $ are independent, that is, two random vectors $ \(
q\psi_1,\dots,q\psi_k \) $ and $ \( (\One-q)\xi_1,\dots,(\One-q)\xi_l \) $
are independent for all $ k,l $ and all $
\psi_1,\dots,\psi_k,\xi_1,\dots,\xi_l \in H^{(1)} $. It follows by
\eqref{1a2} from the similar claim for $ \Down(x_n) $ in place of $ q
$.
\end{proof}

\begin{lemma}\label{1c4}
$ \Up(\Down(x)) = x $ for every $ x \in B $.
\end{lemma}

\begin{proof}
Denote $ q = \Down(x) $, then $ \Down(x') = \One-q $. We have $ \Up(q)
\le x $ (since $ q\psi = Q_x \psi $ is \measurable{\F_x} for $ \psi
\in H^{(1)} $); similarly, $ \Up(\One-q) \le x' $. By \eqref{1c2} and
\eqref{1b8}, $ \Up(q) = x $.
\end{proof}

\begin{lemma}\label{1c5}
If $ q,q_1,q_2,\dots \in \Q $ satisfy $ q_n \uparrow q $ then $
\Up(q_n) \uparrow \Up(q) $.
\end{lemma}

\begin{proof}
$ q_n H^{(1)} \uparrow qH^{(1)} $ implies $ \si(q_n H^{(1)}) \uparrow
\si(qH^{(1)}) $.
\end{proof}

\begin{lemma}\label{1c6}
If $ q,q_1,q_2,\dots \in \Q $ satisfy $ q_n \downarrow q $ then $
\Up(q_n) \downarrow \Up(q) $.
\end{lemma}

\begin{proof}
By \eqref{1b4}, $ \Up(q_n) \downarrow x $ for some $ x \in \La $, $ x \ge
\Up(q) $. By \ref{1c3}, $ \Up(q_n) $ and $ \Up(\One-q_n) $ are independent;
thus, $ x $ and $ \Up(\One-q_n) $ are independent for all $ n $. By
\ref{1c5}, $ \Up(\One-q_n) \uparrow \Up(\One-q) $. Therefore $ x $ and $
\Up(\One-q) $ are independent. By \eqref{1c2} and \eqref{1b8}, $ \Up(q) =
x $.
\end{proof}

Now we are in position to prove that $ C = \Cl(B) $. By
\cite[1d7]{Ts1}, every $ x \in \Cl(B) $ is of the form
\[
x = \liminf_n x_n = \sup_n \inf_k x_{n+k}
\]
for some $ x_n \in B $. It follows that $ \Down(x) = \liminf_n
\Down(x_n) $; by \ref{1c5}, \ref{1c6}, $ \Up(\Down(x)) = \liminf_n
\Up(\Down(x_n)) $; using \ref{1c4} we get $ \Up(\Down(x)) = \liminf_n
x_n = x $.

\begin{sloppypar}
On the other hand, $ \One - \Down(x) = \limsup_n (\One-\Down(x_n)) =
\limsup_n \Down(x'_n) $ by \eqref{1c1}, thus the element $ y =
\Up(\One-\Down(x)) $ satisfies (by \ref{1c5}, \ref{1c6} and \ref{1c4}
again) $ y = \limsup_n \Up(\Down(x'_n)) = \limsup_n x'_n \in \Cl(B)
$.
\end{sloppypar}

By \ref{1c3}, $ x $ and $ y $ are independent. By \eqref{1c2}, $ x
\vee y = 1 $. Therefore, $ y $ is the complement of $ x $ in $ \Cl(B)
$, and we conclude that $ x \in C $. Thus, $ C = \Cl(B) $.

\section[The difficult direction: nonspectral arguments]{\raggedright
  The difficult direction: nonspectral arguments}
\label{sec:2}
\subsection{Preliminaries: finite Boolean algebras}
\label{2a}

Every finite Boolean algebra $ b $ has $ 2^n $ elements, where $ n $
is the number of the atoms $ a_1,\dots,a_n $ of $ b $; these atoms
are pairwise disjoint, and $ { a_1 \vee \dots \vee a_n = 1 }
$. All elements of $ b $ are of the form
\begin{equation}\label{2a1}
a_{i_1} \vee \dots \vee a_{i_k} \, , \quad 1 \le i_1 < \dots < i_k \le
n \, .
\end{equation}
We denote by $ \Atoms(b) $ the set of all atoms of $ b $, and rewrite
\eqref{2a1} as
\begin{equation}\label{2a2}
\forall x \in b \qquad x = \bigvee_{a\in\Atoms(b), a \le x} a \, .
\end{equation}

Let $ B $ be a Boolean algebra, $ b_1,b_2 \subset B $ two finite
Boolean subalgebras, and $ b \subset B $ the Boolean subalgebra
generated by $ b_1 $, $ b_2 $. Then $ b $ is finite. If $ a_1 \in
\Atoms(b_1) $, $ a_2 \in \Atoms(b_2) $ and $ a_1 \wedge a_2 \ne 0 $
then $ a_1 \wedge a_2 \in \Atoms(b) $, and all atoms of $ b $ are of
this form.

\subsection{A random supremum}
\label{2b}

In order to effectively use Condition (c) of Theorem 1 we
choose an increasing sequence $ (x_n)_n $, $ x_n \in B $, whose
supremum is unlike to belong to $ C $. Ultimately it will be proved
that $ \sup_n x_n \in C $ only if $ B $ is classical.

However, we do not construct $ (x_n)_n $ explicitly. Instead we use
probabilistic method: construct a random sequence that has the needed
property with a non-zero probability.

Our noise-type Boolean algebra $ B $ consists of sub-\sif s on a
probability space $ (\Om,\F,P) $. However, randomness of $ x_n $ does
not mean that $ x_n $ is a function on $ \Om $. Another probability
space, unrelated to $ (\Om,\F,P) $, is involved. It may be thought of
as the space of sequences $ (x_n)_n $ endowed with a probability
measure described below.

Often, ``measure on a Boolean algebra $ b $'' stands for an additive
function $ b \to [0,\infty) $. However, the distribution of a random
element of $ b $ (assuming that $ b $ is finite) is rather a
probability measure $ \nu $ on the set of all elements of $ b $, that
is, an additive function $ \nu : 2^b \to [0,\infty) $, $ \nu(b) = 1
$. It boils down to a function $ b \to [0,\infty) $, $ x \mapsto
\nu(\{x\}) $, such that $ \sum_{x\in b} \nu(\{x\}) = 1 $.

Given a finite Boolean algebra $ b $ and a number $ p \in (0,1) $, we
introduce a probability measure $ \nu_{b,p} $ on the set of elements
of $ b $ by
\begin{equation}\label{2b1}
\nu_{b,p} ( \{ a_{i_1} \vee \dots \vee a_{i_k} \} ) = p^k (1-p)^{n-k}
\quad \text{for } 1 \le i_1 < \dots < i_k \le n
\end{equation}
(using the notation of \eqref{2a1}). That is, each atom is included
with probability $ p $, independently of others.

Given finite Boolean subalgebras $ b_1 \subset b_2 \subset \dots
\subset B $ and numbers $ p_1,p_2,\dots \in (0,1) $, we consider
probability measures $ \nu_n = \nu_{b_n,p_n} $ and their product, the
probability measure $ \nu = \nu_1 \times \nu_2 \times \dots $ on the
set $ b_1 \times b_2 \times \dots $ of sequences $ (x_n)_n $, $ x_n
\in b_n $. We note that $ \sup_n x_n \in \Cl(B) $ for all such
sequences and ask, whether
\begin{equation}\label{2b2}
\sup_n x_n \in C \quad \text{for \almost{\nu} all sequences } (x_n)_n
\, ,
\end{equation}
or not.

\begin{proposition}\label{2b3}
If \eqref{2b2} holds for all such $ b_1,b_2,\dots $ and $
p_1,p_2,\dots $ then $ B $ is classical.
\end{proposition}

Condition (c) of Theorem 1 ensures that \eqref{2b2} holds
always. In order to prove the implication (c) \imp (b) (thus
completing the proof of Theorem 1) it is sufficient to prove
Proposition \ref{2b3}.

\subsection{The complement and the tail}
\label{2c}

Let $ x \in \Cl(B) $ be such that $ x_n \uparrow x $ for some $ x_n
\in B $.

\begin{proposition}\label{pr1}
The following five conditions on $ x $ are equivalent.

(a) $ x \in C $;

(b) $ x \vee \lim_n x'_n = 1 $ for some $ x_n \in B $ satisfying $ x_n
\uparrow x $;

(c) $ x \vee \lim_n x'_n = 1 $ for all $ x_n \in B $ satisfying $ x_n
\uparrow x $;

(d) $ \lim_m \lim_n ( x_m \vee x'_n ) = 1 $ for some $ x_n \in B $
satisfying $ x_n \uparrow x $;

(e) $ \lim_m \lim_n ( x_m \vee x'_n ) = 1 $ for all $ x_n \in B $
satisfying $ x_n \uparrow x $.
\end{proposition}

\begin{lemma}\label{le1}
$ (\sup_n x_n) \wedge (\inf_n x'_n) = 0 $ for every increasing
sequence of $ x_n \in B $.
\end{lemma}

\begin{proof}
Note that $ x_m \wedge (\inf_n x'_n) \le x_m \wedge x'_m = 0 $ and use
\cite[(2a6)]{Ts1}.
\end{proof}

\begin{proof}[Proof of Proposition \ref{pr1}]
(c) \imp (b): trivial.

(b) \imp (a): by \ref{le1}, $ x \wedge \lim_n x'_n = 0 $, thus, $
  \lim_n x'_n $ is the complement of $ x $; use \eqref{1b7}.

(a) \imp (c): if $ x_n \uparrow x $ then (taking complements in the
Boolean algebra $ C $) $ x'_n \ge x' $, therefore $ \lim_n x'_n \ge
x' $ and $ x \vee \lim_n x'_n \ge x \vee x' = 1 $.

We see that (a)\equ (b)\equ (c); Lemma \ref{le2} below gives (b)\equ
(d) and (c)\equ (e).
\end{proof}

\begin{lemma}\label{le2}
For every increasing sequence $ (x_n)_n $ of elements of $ B $,
\[
\( \lim_n x_n \) \vee \( \lim_n x'_n \) = \lim_m \lim_n ( x_m \vee
x'_n ) \, .
\]
\end{lemma}

\begin{proof}
Denote for convenience $ y = \lim_n x_n $ and $ z = \lim_n x'_n $. We
have $ x'_n \le x'_m $ for $ n \ge m $. Applying \cite[1d13]{Ts1} to
the pairs $ (x_m,x'_n) \in \La_{x_m} \times \La_{x'_m} $ for a fixed $
m $ and all $ n \ge m $ we get $ x_m \vee x'_n \to x_m \vee z $ as $ n
\to \infty $. Further, $ x_m \wedge z \le x_m \wedge x'_m = 0 $ for
all $ m $; by \cite[(2a6)]{Ts1}, $ y \wedge z = 0 $, and by
\cite[(2a5)]{Ts1}, $ y $ and $ z $ are independent. Applying
\cite[1d13]{Ts1} (again) to $ (x_m,z) \in \La_y \times \La_z $ we get
$ x_m \vee z \to y \vee z $ as $ m \to \infty $. Finally, $ \lim_m
\lim_n ( x_m \vee x'_n ) = \lim_m ( x_m \vee z ) = y \vee z = ( \lim_n
x_n ) \vee ( \lim_n x'_n ) $.
\end{proof}

Given arbitrary (rather than increasing) $ x_n \in B $, we apply
Proposition \ref{pr1} to $ y_n = x_1 \vee \dots \vee x_n $, note that
$ y_n \uparrow \sup_n x_n $ and get
\[
\sup_n x_n \in C \equiv \lim_m \lim_n \( x_1 \vee \dots \vee x_m \vee
( x_1 \vee \dots \vee x_n )' \) = 1 \, .
\]
Thus, \eqref{2b2} is equivalent to
\begin{equation}\label{2c8}
\lim_m \lim_n \( x_1 \vee \dots \vee x_m \vee ( x_1 \vee \dots \vee
x_n )' \) = 1 \quad \text{for \almost{\nu} all sequences } (x_n)_n \,
,
\end{equation}
and Proposition \ref{2b3} is equivalent to the following.

\begin{proposition}\label{2c5}
If \eqref{2c8} holds for all $ b_1 \subset b_2 \subset \dots \subset B
$ and $ p_1,p_2,\dots \in (0,1) $ then $ B $ is classical.
\end{proposition}

In Sect.~\ref{sec:3} we prove Prop.~\ref{2c5}, thus completing the
proof of Theorem 1.

\section[The difficult direction: spectral arguments]{\raggedright The
  difficult direction: spectral arguments}
\label{sec:3}
\subsection{Preliminaries}
\label{3a}

Let $ (S,\Si,\mu) $ be a measure space, $ \mu(S) < \infty $. As usual,
we often treat equivalence classes of measurable functions on $ S $ as
just measurable functions, which is harmless as long as only countably
many equivalence classes are considered simultaneously. Otherwise,
dealing with uncountable sets of equivalence classes, we must be
cautious.

All equivalence classes of measurable functions $ S \to [0,1] $ are a
complete lattice. Let $ Z $ be some set of such classes. If $ Z $ is
countable then its supremum, $ \sup Z $, may be treated naively (as
the pointwise supremum of functions). For an uncountable $ Z $ we have
$ \sup Z = \sup Z_0 $ for some countable $ Z_0 \subset Z $. In
particular, the equality holds whenever $ Z_0 $ is dense in $ Z $
according to the $ L_1 $ metric.

The same holds for functions $ S \to \{0,1\} $ or, equivalently,
measurable sets. Functions $ S \to [0,\infty] $ are also a complete
lattice, since $ [0,\infty] $ can be transformed into $ [0,1] $ by an
increasing bijection.

Another fact. For every increasing sequence of measurable functions $
f_n : S \to [0,\infty) $ there exist $ n_1 < n_2 < \dots $ such that
almost every $ s \in S $ satisfies one of two incompatible conditions:
\begin{equation}\label{3a1}
\begin{gathered}
\text{either} \quad \lim_n f_n(s) < \infty \, , \\
\text{or} \quad f_{n_k}(s) \ge k \text{ for all $ k $ large enough}
\end{gathered}
\end{equation}
(here ``$ k $ large enough'' means $ k \ge k_0(s) $). Proof (sketch):
take $ n_k $ such that
\[
\sum_k \mu \( \{ s : f_{n_k} < k \} \cap \{ s : \lim_n f_n(s) = \infty
\} \) < \infty \, .
\]

All said above holds also for a measure class space $ (S,\Si,\M) $
(see \cite[Sect.~2a]{Ts2}) in place of the measure space $ (S,\Si,\mu)
$.

\subsection{Classicality via spectrum}
\label{3b}

We have to prove Prop.~\ref{2c5}. A similar result
\cite[Sect.~6c/6.3]{Ts03} in the older framework (of noises over $ \R
$) was proved via spectral sets: almost all spectral sets are finite,
therefore the noise is classical. In the new framework (Boolean
algebras of \sif s), spectral sets appear only in special situations
\cite[Sects.~2c, 3b]{Ts2}, not in general. Nevertheless, in general we
can define on the spectral space $ S $ a function $ K $ that plays the
role of the number of spectral points, and deduce classicality from
almost everywhere finiteness of $ K $.

Recall \cite[Sect.~2b]{Ts2}: the spectrum of a noise-type Boolean
algebra $ B $ is a measure class space $ (S,\Si,\M) $ together with an
isomorphism $ \al : \A \to L_\infty(S,\Si,\M) $, where $ \A $ is the
commutative von Neumann algebra generated by the operators $ Q_x $
(for $ x \in B $) on the Hilbert space $ H = L_2(\Om,\F,P) $. Recall
also $ S_x $: $ \al(Q_x) = \One_{S_x} $ for $ x \in \Cl(B) $.

Claim:
\begin{equation}\label{3b01}
x_n \downarrow x \text{ implies } S_{x_n} \downarrow S_x \, ;
\; \text{ also } x_n \uparrow x \text{ implies } S_{x_n} \uparrow S_x
\, ;
\end{equation}
here $ x,x_1,x_2,\dots \in \Cl(B) $. Proof: let $ x_n \downarrow x $,
then $ Q_{x_n} \downarrow Q_x $ by \eqref{1b3}, thus $ \al(Q_{x_n})
\downarrow \al(Q_x) $, which means $ S_{x_n} \downarrow S_x $; the
case $ x_n \uparrow x $ is similar.

\begin{sloppypar}
The relation $ \lim_m \lim_n ( x_m \vee x'_n ) = 1 $ treated in
Sect.~\ref{2c} becomes $ { \lim_m \lim_n S_{x_m \vee x'_n} = S } $, that
is, $ \cup_m \cap_n S_{x_m \vee x'_n} = S $; in other words, almost
every $ s \in S $ satisfies $ \exists m \, \forall n \;\> s \in S_{x_m
\vee x'_n} $. Accordingly, \eqref{2c8} may be rewritten as follows:
\begin{multline}\label{3b02}
\text{for \almost{\nu} all sequences } (x_n)_n \, , \quad
 \text{for almost all } s \in S \, , \\
\exists m \; \forall n \quad s \in S_{ x_1 \vee \dots \vee x_m \vee (
 x_1 \vee \dots \vee x_n )' } \, .
\end{multline}
\end{sloppypar}

Let $ b \subset B $ be a finite Boolean subalgebra. For every $ s \in
S $ the filter $ \{ x \in b : s \in S_x \} $ (a kind of rough
``spectral filter'', see \cite[2c]{Ts2}), like every filter on a
finite Boolean algebra, is generated by some $ x_b(s) \in b $:
\begin{equation}\label{3b03}
\forall x \in b \quad \( s \in S_x \equiv x \ge x_b(s) \) \, .
\end{equation}
Like every element of $ b $, $ x_b(s) $ is the union of some of the
atoms of $ b $; the number of these atoms will be denoted by $ K_b(s)
$:
\[
K_b(s) = | \{ a \in \Atoms(b) : a \le x_b(s) \} | \, .
\]
For two finite Boolean subalgebras,
\[
\text{if} \quad b_1 \subset b_2 \quad \text{then} \quad K_{b_1}(\cdot)
\le K_{b_2}(\cdot) \, .
\]
Each $ K_b $ is an equivalence class (rather than a function), and the
set of all $ b $ need not be countable. We take supremum in the
complete lattice of all equivalence classes of measurable functions $
S \to [0,+\infty] $:
\[
K = \sup_b K_b \, , \quad K : S \to [0,+\infty] \, ,
\]
where $ b $ runs over all finite Boolean subalgebras $ b \subset B $.

Proposition \ref{2c5} follows immediately from Proposition \ref{3b1}
and Theorem \ref{3b2} below.

\begin{proposition}\label{3b1}
If \eqref{3b02} holds for all $ b_1 \subset b_2 \subset \dots $ and $
p_1,p_2,\dots \in (0,1) $ then $ K(\cdot) < \infty $ almost
everywhere.
\end{proposition}

\begin{theorem}\label{3b2}
If $ K(\cdot) < \infty $ almost everywhere then $ B $ is classical.
\end{theorem}

\subsection{Proving Proposition \ref{3b1}}
\label{3c}

We choose $ p_1,p_2,\dots \in (0,1) $ and $ c_1,c_2,\dots \in
\{1,2,3,\dots\} $ such that
\begin{gather}
\sum_n p_n < 1 \, , \label{3c1} \\
(1-p_n)^{c_n} \to 0 \quad \text{as } n \to \infty \, . \label{3c2}
\end{gather}
We also choose finite Boolean subalgebras $ b_1 \subset b_2 \subset
\dots \subset B $ such that $ K_{b_n} \uparrow K $.

We define for $ x \in B $
\begin{gather}
K_b(x,s) = | \{ a \in \Atoms(b) : a \le x_b(s) \wedge x \} | \, , \\
K(x) = \sup_b K_b(x) \, ;
\end{gather}
as before, $ b $ runs over all finite Boolean subalgebras of $ B $; $
K_b(x) = K_b(x,\cdot) $ is an equivalence class, and $ K(x) $ is the
supremum of such equivalence classes. (Hopefully, the new $ K(x) $
will not be confused with the old $ K(s) $.)

Claim (additivity):
\begin{equation}\label{3c25}
K(x\vee y) = K(x) + K(y) \quad \text{whenever } x \wedge y = 0 \, .
\end{equation}
Proof: $ K_b (x\vee y) = K_b (x) + K_b (y) $ when $ b \ni x,y $; this
restriction on $ b $ does not change the supremum.

Claim:
\begin{equation}
K_{b_n}(x) \uparrow K(x)
\end{equation}
for every $ x \in B $. Proof: $ K(x) \ge \lim_n K_{b_n}(x) = \lim_n \(
K_{b_n} - K_{b_n}(x') \) \ge K - K(x') = K(x) $.

Using \eqref{3a1} we can take $ n_1 < n_2 < \dots $ such that
for almost every $ s \in S $
\begin{equation}\label{3c3}
\begin{gathered}
\text{either } K(x,s) < \infty \, , \\
\text{or } K_{b_{n_k}}(x,s) \ge c_k \text{ for all $ k $ large enough.}
\end{gathered}
\end{equation}
These $ n_k $ depend on $ x \in B $. However, countably many $ x $ can
be served by a single sequence $ (n_k)_k $ using the well-known
diagonal argument. This way we ensure \eqref{3c3} with a single $
(n_k)_k $ for all $ x \in b_1 \cup b_2 \cup \dots\, $ Now we rename $
b_{n_k} $ into $ b_k $, discard a null set of bad points $ s \in S $
and get
\begin{equation}\label{3c4}
\begin{gathered}
\text{either } K(x,s) < \infty \, , \\
\text{or } K_{b_n}(x,s) \ge c_n \text{ for all $ n $ large enough}
\end{gathered}
\end{equation}
for all $ x \in b_1 \cup b_2 \cup \dots $ and $ s \in S $; here ``$n$
large enough'' means $ n \ge n_0(x,s) $.

We recall the product measure $ \nu = \nu_1 \times \nu_2 \times \dots
$ introduced in Sect.~\ref{2b} on the product set $ b_1 \times b_2
\times \dots $; as before, $ \nu_n = \nu_{b_n,p_n} $.
For notational convenience we treat the coordinate maps $ X_n : ( b_1
\times b_2 \times \dots, \nu ) \to b_n $, $ X_n(x_1,x_2,\dots) = x_n
$, as independent \valued{b_n} random variables; $ X_n $ is
distributed $ \nu_n $, that is, $ \Pr{ X_n = x } = \nu_n(\{x\}) $ for
$ x \in b_n $. We introduce \valued{b_n} random variables
\[
Y_n = X_1 \vee \dots \vee X_n \, .
\]

\begin{lemma}\label{3c6}
$ \Pr{ K(Y'_n,s) < \infty } \le p_1 + \dots + p_n $ for all $ s \in S
$ such that $ K(s) = \infty $ and all $ n $.
\end{lemma}

\begin{proof}
There exists $ a \in \Atoms(b_n) $ such that $ K(a,s) = \infty $
(since $ \sum_a K(a,s) = K(s) = \infty $). We have $ K(Y'_n,s) <
\infty \imply a \le Y_n \imply \exists k\in\{1,\dots,n\} \;\> a \le
X_k $, therefore $ \Pr{ K(Y'_n,s) < \infty } \le \sum_{k=1}^n \Pr{ a
\le X_k } = \sum_{k=1}^n p_k $.
\end{proof}

\begin{lemma}\label{3c7}
If $ x \in b_m $ and $ s \in S $ satisfy $ K(x,s)=\infty $ then
\[
\Pr{ \forall n > m \quad X_n \wedge x \wedge x_{b_n}(s) = 0 } = 0 \,
.
\]
\end{lemma}

\begin{proof}
For $ n > m $,
\[
\Pr{ X_n \wedge x \wedge x_{b_n}(s) = 0 } = (1-p_n)^{K_{b_n}(x,s)} \,
,
\]
since $ x \wedge x_{b_n}(s) $ contains $ K_{b_n}(x,s) $ atoms of $ b_n
$. By \eqref{3c4}, $ K_{b_n}(x,s) \ge c_n $ for all $ n $ large
enough. Thus, $ \Pr{ X_n \wedge x \wedge x_{b_n}(s) = 0 } \le
(1-p_n)^{c_n} \to 0 $ as $ n \to \infty $ by \eqref{3c2}.
\end{proof}

\begin{lemma}\label{3c8}
\[
\Pr{ K(Y'_m,s)=\infty \text{ \small and } \forall n > m \;\> s \in
S_{Y_m\vee Y'_n} } = 0
\]
for all $ s \in S $ and $ m $.
\end{lemma}

\begin{proof}
By \eqref{3b03}, $ s \in S_{Y_m\vee Y'_n} \equiv Y_m\vee Y'_n \ge
x_{b_n}(s) $. We have to prove that
\[
\Pr{ Y'_m=y \text{ \small and } \forall n > m \;\> Y_m\vee Y'_n \ge
x_{b_n}(s) } = 0
\]
for every $ y \in b_m $ satisfying $ K(y,s)=\infty $. By \ref{3c7},
\[
\Pr{ \forall n > m \;\> X_n \wedge y \wedge x_{b_n}(s) = 0 } = 0 \,
.
\]
It remains to note that $ y' \vee Y'_n \ge x_{b_n}(s) \equiv ( y
\wedge Y_n )' \ge x_{b_n}(s) \equiv ( y \wedge Y_n ) \wedge x_{b_n}(s)
= 0 \imply y \wedge X_n \wedge x_{b_n}(s) = 0 $.
\end{proof}

Now we prove Proposition \ref{3b1}. Applying \eqref{3b02} to $
b_1,b_2,\dots $ and $ p_1,p_2,\dots $ chosen in Sect.~\ref{3c} we have
\[
\exists m \;\; \forall n \quad s \in S_{Y_m\vee Y'_n}
\]
almost surely, for almost all $ s \in S $. In combination with
\ref{3c8} it gives
\[
\Pr{ \exists m \;\> K(Y'_m,s)<\infty } = 1
\]
for almost all $ s \in S $. On the other hand, by \ref{3c6} and
\eqref{3c1},
\[
\Pr{ \exists m \;\> K(Y'_m,s)<\infty } \ne 1
\]
for all $ s \in S $ such that $ K(s)=\infty $. Therefore $ K(s)<\infty
$ for almost all $ s $, which completes the proof of \ref{3b1}.

\subsection{Proving Theorem \ref{3b2}}
\label{3d}

Recall \cite[Sect.~2b]{Ts2}: every measurable set $ E \subset S $
corresponds to a subspace $ H(E) \subset H $ by $
\al(\operatorname{Pr}_{H(E)}) = \One_E $. Additivity holds: $ H ( E_1
\uplus E_2 ) = H(E_1) \oplus H(E_2) $; countable additivity also
holds.

In particular, consider sets $ E_k = \{ s \in S : K(s) = k \} $.

\begin{proposition}\label{3d1}
$ H(E_1) = H^{(1)} $ (the first chaos space defined by
\cite[1a2]{Ts2}).
\end{proposition}

We define chaos spaces $ H^{(k)} = H(E_k) $.

Recall that classicality is defined as the equality $ \si(H^{(1)}) =
\F_1 $ \cite[1a2]{Ts2}.

\begin{proposition}\label{3d2}
$ \si(H^{(k)}) \subset \si(H^{(1)}) $ for all $ k=2,3,\dots $
\end{proposition}

If $ K < \infty $ almost everywhere then $ S = \uplus_k E_k $, thus $
H = \oplus_k H^{(k)} $, therefore $ \F_1 = \bigvee_k \si(H^{(k)}) =
\si(H^{(1)}) $, which means that $ B $ is classical. We see that
Theorem \ref{3b2} follows from Propositions \ref{3d1}, \ref{3d2}.

We denote $ b_x = \{ 0, x, x', 1 \} $ for $ x \in B $.

\begin{lemma}\label{3d3}
For every $ x \in B $,
\[
\{ \psi \in H : \psi = Q_x \psi + Q_{x'} \psi \} = H \( \{ s :
K_{b_x}(s) = 1 \} \) \, .
\]
\end{lemma}

\begin{proof}
$ \{ s : K_{b_x}(s) = 1 \} = \{ s : K_{b_x}(s) \le 1 \}
\setminus \{ s : K_{b_x}(s) = 0 \} = ( S_x \cup S_{x'} ) \setminus S_0
= ( S_x \setminus S_0 ) \uplus ( S_{x'} \setminus S_0 ) $ (since $ S_x
\cap S_{x'} = S_0 $), thus $ H \( \{ s : K_{b_x}(s) = 1 \} \) = H (
S_x \setminus S_0 ) \oplus H ( S_{x'} \setminus S_0 ) = ( H_x \ominus
H_0 ) \oplus ( H_{x'} \ominus H_0 ) $; use \cite[1b6]{Ts2}.
\end{proof}

\begin{lemma}\label{3d4}
Assume that $ b_1, b_2 \subset B $ are finite Boolean subalgebras, and
$ b \subset B $ is the (finite) Boolean subalgebra generated by $ b_1,
b_2 $. Then
\[
\{ s : K_{b_1}(s) \le 1 \} \cap \{ s : K_{b_2}(s) \le 1 \} \subset \{
s : K_b(s) \le 1 \} \, .
\]
\end{lemma}

\begin{proof}
If $ K_{b_1}(s) \le 1 $, $ K_{b_2}(s) \le 1 $ and $ s \notin S_0 $
then $ x_{b_1}(s) \in \Atoms(b_1) $, $ x_{b_2}(s) \in \Atoms(b_2) $,
thus $ x_b(s) \le x_{b_1}(s) \wedge x_{b_2}(s) \in \Atoms(b) $ (recall
Sect.~\ref{2a}), therefore $ K_b(s) \le 1 $.
\end{proof}

\begin{lemma}\label{3d5}
\begin{gather*}
\{ s : K(s) \le 1 \} = \bigcap_{x\in B} \{ s : K_{b_x}(s) \le 1 \} \,
; \\
\{ s : K(s) = 1 \} = \bigcap_{x\in B} \{ s : K_{b_x}(s) = 1 \} \,
.
\end{gather*}
\end{lemma}

\begin{proof}
Every finite Boolean subalgebra $ b $ is generated by four-element
Boolean subalgebras $ b_x $ for $ x \in b $; by \ref{3d4}, $ \{ s :
K_b(s) \le 1 \} \supset \bigcap_{x\in b} \{ s : K_{b_x}(s) \le 1 \} $;
the intersection over all $ b $ gives $ \{ s : K(s) \le 1 \} \supset
\bigcap_{x\in B} \{ s : K_{b_x}(s) \le 1 \} $. The converse inclusion
being trivial, we get the first equality. The second equality follows,
since the set $ \{ s : K_b(s) = 0 \} $ is equal to $ S_0 $,
irrespective of $ b $.
\end{proof}

\begin{proof}[Proof of Proposition \ref{3d1}]
It follows from the second equality of \ref{3d5} that $ H(E_1) =
\bigcap_x H \( \{ s : K_{b_x}(s) = 1 \} \) $. Using \ref{3d3} we get $
\bigcap_x \{ \psi \in H : \psi = Q_x \psi + Q_{x'} \psi \} $, that is, $
\{ \psi \in H : \forall x \; \psi = Q_x \psi + Q_{x'} \psi \} $. This
set contains $ H^{(1)} $ evidently, and is contained in $ H^{(1)} $ by
\cite[1b4]{Ts2}.
\end{proof}

The proof of Proposition \ref{3d2} is postponed to Sect.~\ref{3f}.

\subsection{More on spectrum and factorization}
\label{3e}

Let $ B $ be a noise-type Boolean algebra of \sif s on a probability
space $ (\Om,\F,P) $, and $ x \in B $. Then the set $ B_x = \{ u \in B
: u \le x \} $ is also a Boolean algebra (and a sublattice of $ B $,
but not a Boolean subalgebra of $ B $). Moreover, $ B_x $ is a
noise-type Boolean algebra of \sif s on the probability space $
(\Om,\F_x,P|_{\F_x}) $. Thus, notions introduced for $ B $ have their
counterparts for $ B_x $, listed below (up to natural isomorphisms):
\[
\begin{matrix}
H = L_2(\F) & H_x = L_2(\F_x) \\
Q_y \text{ for } y \in B & Q_u|_{H_x} \text{ for } u \in B_x \\
\al(\A) = L_\infty(\Si) & \al(\A_x) = L_\infty(\Si_x) \\
H(E) = \al^{-1}(\One_E)H \subset H \;\; \text{\small for $E \in \Si$} \;\;\; &
 H_x(E) = \al^{-1}(\One_E)H_x \subset H_x \;\; \text{\small for $E \in
 \Si_x$} \\
H^{(1)} & Q_x H^{(1)} = H^{(1)} \cap H_x \\
s \mapsto K(s) \;\; \text{\small ($\Si$-measurable)} & s \mapsto
 K(x,s) \;\; \text{\small ($\Si_x$-measurable)}
\end{matrix}
\]

We have $ B = B_x \times B_{x'} $ (the product of Boolean algebras)
up to the natural isomorphism: $ B_x \times B_{x'} \ni (u,v) \mapsto u
\vee v \in B $, that is, $ B \ni y \mapsto ( y \wedge x, y \wedge x' )
\in B_x \times B_{x'} $. Also, $ \F = \F_x \otimes \F_{x'} $ (the
\sif\ generated by two independent sub-\sif s), $ H = H_x \otimes
H_{x'} $, and $ \B(H) = \B(H_x) \otimes \B(H_{x'}) $ (the algebra of
all bounded operators). The algebra $ \B(H_x) $ is naturally embedded
into $ \B(H) $ by $ \B(H_x) \ni T \mapsto T \otimes \One \in \B(H) $,
that is, $ (T\otimes\One)(fg) = (Tf) g $ for $ f \in H_x $, $ g \in
H_{x'} $. In particular, if $ T = Q_u|_{H_x} $ for some $ u \in B_x $
then $ T\otimes\One = Q_{u\vee x'} $ (rather than $ Q_u $), which
motivates the following approach (dual to the approach via $ B_x
\times B_{x'} $).

The set $ B^x = \{ w \in B : w \ge x' \} $ is another Boolean algebra,
naturally isomorphic to $ B_x $: $ B_x \ni u \mapsto u \vee x' \in B^x
$, that is, $ B^x \ni w \mapsto w \wedge x \in B_x $. Accordingly, $ B
= B^x \times B^{x'} $ up to another natural isomorphism: $ B^x \times
B^{x'} \ni (w,z) \mapsto w \wedge z \in B $, that is, $ B \ni y
\mapsto ( y \vee x', y \vee x ) \in B^x \times B^{x'} $. Note that $ 
Q_{w\wedge z} = Q_w Q_z $ (while generally $ Q_{u\vee v} \ne Q_u + Q_v
- Q_u Q_v $). Note also that $ \A_x $ is generated by $ \{ Q_y : x
\vee y = 1 \} = \{ Q_w : w \in B^x \} $ \cite[Sect.~2b]{Ts2}, and $ \A
= \A_x \otimes \A_{x'} $.
For $ E_1 \in \Si_x $, $ E_2 \in \Si_{x'} $ we have $ \al^{-1} (
\One_{E_1\cap E_2} ) = \al^{-1} ( \One_{E_1} \One_{E_2} ) =
\al^{-1}(\One_{E_1}) \al^{-1}(\One_{E_2}) $, thus,
\begin{equation}\label{3e1}
H ( E_1 \cap E_2 ) = H_x (E_1) \otimes H_{x'} (E_2) \quad \text{for }
E_1 \in \Si_x, \, E_2 \in \Si_{x'} \, .
\end{equation}

Results proved for $ B $ have their counterparts for $ B_x $. In
particular, here is the counterpart of Proposition \ref{3d1}.

\begin{proposition}\label{3e2}
$ H_x \( \{ s \in S : K(x,s)=1 \} \) = H_x \cap H^{(1)} $.
\end{proposition}

\subsection{Proving Proposition \ref{3d2}}
\label{3f}

\begin{lemma}\label{3f1}
$ \{ s \in S : K(s) = 2 \} = \sup_{x\in B} \{ s \in S : K(x,s) =
K(x',s) = 1 \} $ (the supremum of equivalence classes).
\end{lemma}

\begin{proof}
``$\supset$'' follows from \eqref{3c25}; it is sufficient to prove that
$ \{ s \in S : K(s) = 2 \} \subset \cup_{x\in b_1\cup b_2\cup\dots} \{
s \in S : K(x,s) = K(x',s) = 1 \} $ if $ b_1 \subset b_2 \subset \dots
$ satisfy $ K_{b_n} \uparrow K $.

Given $ s $ such that $ K(s)=2 $, we take $ n $ such that $ K_{b_n}(s)
= 2 $, that is, $ x_{b_n}(s) $ contains exactly two atoms of $ b_n
$. We choose $ x \in b_n $ that contains exactly one of these two
atoms; then $ K_{b_n}(x,s) = K_{b_n}(x',s) = 1 $, therefore $ K(x,s) =
K(x',s) = 1 $, since $ 1 = K_{b_n}(x,s) \le K(x,s) = K(s) - K(x',s)
\le 2 - K_{b_n}(x',s) = 1 $.
\end{proof}

\begin{proof}[Proof of Proposition \ref{3d2} for $ k=2 $]
It follows from \ref{3f1} that $ H^{(2)} $ is generated (as a closed
linear subspace of $ H $) by the union, over all $ x \in B $, of the
subspaces $ H \( \{ s \in S : K(x,s) = K(x',s) = 1 \} \) $. In order
to get $ \si(H^{(2)}) \subset \si(H^{(1)}) $ it is sufficient to prove
that
\begin{equation}\label{3f3}
\si \( H \( \{ s \in S : K(x,s) = K(x',s) = 1 \} \) \) \subset
\si(H^{(1)}) \quad \text{for all } x \, .
\end{equation}
By \eqref{3e1} and \ref{3e2}, $ H \( \{ s \in S : K(x,s) = K(x',s) = 1
\} \) = H_x \( \{ s \in S : K(x,s) = 1 \} \) \otimes H_{x'} \( \{ s
\in S : K(x',s) = 1 \} \) = ( H_x \cap H^{(1)} ) \otimes ( H_{x'} \cap
H^{(1)} ) $, which implies \eqref{3f3}.
\end{proof}

The proof of Proposition \ref{3d2} for higher $ k $ is similar. Lemma
\ref{3f1} is generalized to
\[
\{ s \in S : K(s) = k \} = \sup_{x\in B} \{ s \in S : K(x,s) = k-1, \,
K(x',s) = 1 \} \, ,
\]
and \eqref{3f3} --- to
\[
\si \( H \( \{ s \in S : K(x,s) = k-1, \, K(x',s) = 1 \} \) \) \subset
\si(H^{(k-1)} \cup H^{(1)}) \, .
\]
Thus, $ \si(H^{(k)}) \subset \si(H^{(k-1)} \cup H^{(1)}) $. By
induction in $ k $, $ \si(H^{(k)}) \subset \si(H^{(1)}) $, which
completes the proof of Proposition \ref{3d2}, Theorem \ref{3b2},
Propositions \ref{2c5} and \ref{2b3} and finally, Theorem 1.

\bigskip
\filbreak
{
\small
\begin{sc}
\parindent=0pt\baselineskip=12pt
\parbox{4in}{
Boris Tsirelson\\
School of Mathematics\\
Tel Aviv University\\
Tel Aviv 69978, Israel
\smallskip
\par\quad\href{mailto:tsirel@post.tau.ac.il}{\tt
 mailto:tsirel@post.tau.ac.il}
\par\quad\href{http://www.tau.ac.il/~tsirel/}{\tt
 http://www.tau.ac.il/\textasciitilde tsirel/}
}

\end{sc}
}
\filbreak

\end{document}